\documentclass[10pt,a4paper]{article}
\usepackage[utf8]{inputenc}
\usepackage[T1]{fontenc}
\usepackage{amsmath}
\usepackage{amsfonts}
\usepackage{amssymb}
\usepackage{amsthm}

\theoremstyle{definition}
\newtheorem{definition}{Definition}
\newtheorem{example}[definition]{Example}
\newtheorem*{notation}{Notation}

\theoremstyle{remark}
\newtheorem{remark}[definition]{Remark}
\newtheorem*{idea}{Main idea}

\theoremstyle{plain}
\newtheorem{theorem}[definition]{Theorem}
\newtheorem{proposition}[definition]{Proposition}
\newtheorem{corollary}[definition]{Corollary}
\newtheorem*{theorem_intro}{Theorem}

\DeclareMathOperator{\Liminf}{lim\ inf\ }
\DeclareMathOperator{\Limsup}{lim\ sup\ }
\DeclareMathOperator{\densinf}{\underline{dens}}
\DeclareMathOperator{\denssup}{\overline{dens}}
\DeclareMathOperator{\expo}{\textnormal{e}}

\begin{document}

\title{Common upper frequent hypercyclicity}

\author{Monia Mestiri \footnote{The author is supported by a grant of F.R.S.-FNRS.}}

\date{}

\maketitle

\begin{abstract}
Considering a family of upper frequently hypercyclic operators we care about the existence of vectors which are upper frequently hypercyclic for any operator of this family. We establish sufficient conditions for a family of operators to have these vectors called common upper frequently hypercyclic vectors. Using this result, we then give a construction of such vectors. Finally we derive some applications to families of weighted shifts.
\end{abstract}

\section{Introduction}

In the last decades the notion of hypercyclicity has been subject to many investigations. An operator $T$ on a Fr\'echet space $X$ is called \textit{hypercyclic} on $X$ if there exists some vector $x \in X$ such that the set $\lbrace T^n(x) \mid n \geq 0 \rbrace$ is dense in $X.$ In this case, $x$ is called \textit{hypercyclic} for $T.$ Birkhoff has established an equivalent formulation of hypercyclicity in \cite{Bir-22}. In addition he proved that the set of hypercyclic vectors for a hypercyclic operator is a dense $G_\delta$-set. In 1969 Rolewicz has found one of the first example of hypercyclic operators: the multiples of the backward shift $\lambda B,$ with $\lambda > 1$ on $l^p.$

Later Salas raised in \cite{Sal-99} the question of the existence of vectors which are hypercyclic for any multiple of the backward shift $\lambda B,\ \lambda > 1.$ This initiated a new notion: common hypercyclicity. For an arbitrary family of hypercyclic operators $(T(\lambda))_{\lambda \in \Lambda},$ a vector is then called a \textit{common hypercyclic vector} for $(T(\lambda))_{\lambda \in \Lambda}$ if it is hypercyclic for each operator of the family. The Birkhoff theorem combined with the Baire category theorem directly implies that any countable family of hypercyclic operators has a common hypercyclic vector. Actually in this case the set of common hypercyclic vectors is a dense $G_\delta$-set. For uncountable families of operators this argument fails. Indeed even if Abakumov and Gordon have shown in \cite{AbGo-03} the existence of common hypercyclic vectors for $(\lambda B)_{\lambda > 1},$ there also exist families of operators without common hypercyclic vectors. For example Bayart and Matheron have proved in \cite{BaMa-09} that the family of hypercyclic weighted shifts on $l^2$ does not possess any common hypercyclic vector. However imposing structures on the index set, positive results have been obtained (see  \cite{BaMa-07}, \cite{CoSa-04-a}, \cite{Shk-10}). Most theorems on common hypercyclicity are based on the following generalization of the Birkhoff theorem for families of hypercyclic operators. The first part of this result is due to Saint-Raymond (see \cite{AbGo-03}, \cite{Bay-04}). The equivalence has been found by Shkarin in \cite{Shk-10}.

\begin{theorem_intro}[Shkarin \cite{Shk-10}]
Let $\Lambda$ be a $\sigma$-compact metric space, $X$ a separable Fr\'echet space and $(T(\lambda))_{\lambda \in \Lambda}$ a continuous family of operators on $X.$ Then the set of common hypercyclic vectors for $(T(\lambda))_{\lambda \in \Lambda}$ is a $G_\delta$-set. Moreover the following assertions are equivalent:
\begin{itemize}
\item[\textnormal{(i)}] the set of common hypercyclic vectors for $(T(\lambda))_{\lambda \in \Lambda}$ is a dense $G_\delta$-set;
\item[\textnormal{(ii)}] for any compact subset $K$ of $\Lambda,$ for any non-empty open subsets $U, V$ of $X,$ there exists some $x \in U$ such that
\[
\forall\ \lambda \in K,
\ \exists\ n \geq 0,
\ T^n(\lambda)(x) \in V.
\] 
\end{itemize}
\end{theorem_intro}  

The question of the existence of common hypercyclic vectors extends naturally to frequent hypercyclicity, a notion introduced in \cite{BaGr-06}. Recall that the \textit{lower density} of a subset $A$ of $\mathbb{N}_0$ is given by
\[
\densinf (A) = 
\underset{n \to \infty}{\Liminf}
\frac{\#(A \cap [0, n])}{n + 1}.
\]
On the other hand the \textit{upper density} of a subset $A$ of $\mathbb{N}_0$ is given by
\[
\denssup (A) = 
\underset{n \to \infty}{\Limsup}
\frac{\#(A \cap [0, n])}{n + 1}.
\] 
An operator $T$ on $X$ is then called \textit{frequently hypercyclic} if there exists some vector $x \in X$ such that for any non-empty open subset $U$ of $X,$ we have
\[
\densinf (\lbrace m \geq 0 \mid T^m(x) \in U \rbrace) > 0.
\]
In this case, the vector $x$ is called \textit{frequently hypercyclic} for $T.$ In contrast to hypercyclicity, Moothathu in \cite{Moo-13}, Grivaux and Matheron in \cite{GrMa-14} and Bayart and Ruzsa in \cite{BaRu-15} have proved that the set of frequently hypercyclic vectors for an operator is of first category. Therefore the methods used for common hypercyclicity fail in this case. We then care about a weaker notion introduced by Shkarin in \cite{Shk-09}, the upper frequent hypercyclicity. An operator $T$ on $X$ is called \textit{upper frequently hypercyclic} if there exists some vector $x \in X$ such that for any non-empty open subset $U$ of $X,$
\[
\denssup (\lbrace m \geq 0 \mid T^m(x) \in U \rbrace) > 0.
\]
In this case the vector $x$ is called \textit{upper frequently hypercyclic} for $T.$ The set of upper frequently hypercyclic vectors is denoted by $\mathcal{U}FHC(T).$ The properties of this set is more similar to the set of hypercyclic vectors. Indeed Bayart and Ruzsa have established in \cite{BaRu-15} that, for an upper frequently hypercyclic operator, this set is residual. Furthermore, Bonilla and Grosse-Erdmann have obtained in \cite{BoGr-16} the following theorem analogue to the Birkhoff theorem for upper frequent hypercyclicity.

\begin{theorem_intro}[Bonilla - Grosse-Erdmann \cite{BoGr-16}]
Let $X$ be a separable Fr\'echet space and $T$ an operator on $X.$ Then the following assertions are equivalent:
\begin{itemize}
\item[\textnormal{(i)}] the set of upper frequently hypercyclic vectors for $T$ is residual in $X$; 
\item[\textnormal{(ii)}] for any non-empty open subset $V$ of $X,$ there exists some $\delta > 0$ such that for any non-empty open subset $U$ of $X$ and any $N \in \mathbb{N},$ we have 
\[
\exists\ x \in U,
\ \exists\ n \geq N,
\ \frac{\#\lbrace 0 \leq m \leq n \mid T^m(x) \in V \rbrace}{n + 1} > \delta.
\]
\end{itemize}
\end{theorem_intro}

Using this result we will obtain sufficient conditions for common upper frequent hypercyclicity. This yields the following theorem which acts as the generalization of the Birkhoff theorem.

\begin{theorem}\label{thm:Basic criterion}
Let $\Lambda$ be a $\sigma$-compact metric space, $X$ a separable Fr\'echet space and $(T(\lambda))_{\lambda \in \Lambda}$ a continuous family of operators on $X.$ Suppose that for any non-empty open subset $V$ of $X$ and for any compact subset $K$ of $\Lambda$ there exists some $\delta > 0$ such that, for any non-empty open subset $U$ of $X$ and any $M \in \mathbb{N},$ we have
\[
\exists\ x \in U,
\ \forall\ \lambda \in K,
\ \exists\ n \geq M,
\ \frac{\# \lbrace 0 \leq m \leq n \mid T^m(\lambda)(x) \in V \rbrace}{n + 1}
> \delta.
\]
Then the set of common upper frequently hypercyclic vectors for $(T(\lambda))_{\lambda \in \Lambda}$ is residual in $X.$
\end{theorem}

Thanks to this result, in order to prove the existence of common upper frequently hypercyclic vectors it suffices for each non-empty open subsets $U, V$ of $X $ and for each compact subset $K$ of $\Lambda$ to create a vector of $U$ that visits frequently enough $V$ under $T(\lambda),$ for any $\lambda \in K.$ Adapting methods used for common hypercyclicity we will give a construction of such a vector in Section~2. We will then study applications to families of weighted shifts. 

\section{A generalization of the Birkhoff theorem for common upper frequent hypercyclicity}

Our results take place in the same context as common hypercyclicity criteria. In other terms we consider families of operators indexed by a $\sigma$-compact metric space. We then introduce an additional continuity for the family of operators. A family of operators $(T(\lambda))_{\lambda \in \Lambda}$ indexed by a metric space is called \textit{continuous} if for any $x \in X,$ the map $\Lambda \to X : \lambda \mapsto T(\lambda)(x)$ is continuous on $\Lambda.$ 
In this context, we have the two following results from \cite[Chapter 11]{GrPe-11}.

\begin{proposition}\label{prop:continuous_family}
Let $\Lambda$ be a metric space, $X$ a Fr\'echet space and $(T(\lambda))_{\lambda \in \Lambda}$ a continuous family of operators on $X.$ The map $\Lambda \times X \to X : (\lambda, x) \mapsto T(\lambda)(x)$ is then continuous on $\Lambda \times X$ with respect to the product topology.
\end{proposition}

\begin{corollary}\label{coro:continuous_family}
Let $\Lambda$ be a metric space, $X$ a Fr\'echet space and $(T(\lambda))_{\lambda \in \Lambda}$ a continuous family of operators on $X.$ Then for any $n \geq 0,$ the family of operators $(T^n(\lambda))_{\lambda \in \Lambda}$ is continuous.
\end{corollary}

As announced in the introduction we start with Theorem \ref{thm:Basic criterion} which acts as the generalization of the Birkhoff theorem for common hypercyclicity. Before stating it we have to introduce the following notation.

\begin{notation}
Let $X$ be a Fr\'echet space, $(T(\lambda))_{\lambda \in \Lambda}$ a family of operators on $X$ and $V$ an open subset of $X.$ For each $\lambda \in \Lambda$ and each $x \in X,$ we denote by $N_\lambda(x, V)$ the set given by
\[
N_\lambda(x, V) :=
\lbrace
n \geq 0
\mid
T^n(\lambda)(x) \in V
\rbrace.
\]
This is the set of visiting times for $x$ in $V$ under the operator $T(\lambda)$.
\end{notation} 

\begin{proof}[Proof of Theorem \ref{thm:Basic criterion}]
By separability of $X$ there exists some countable base $(V_k)_{k \geq 1}$  of non-empty open subsets of $X.$ Moreover the $\sigma$-compactness of $\Lambda$ ensures the existence of a sequence $(K_m)_{m \geq 1}$ of compact subsets of $\Lambda$ such that
\begin{equation}
\Lambda = \underset{m \geq 1}{\bigcup} K_m.
\label{eq:Basic_criterion_1}
\end{equation}
From the hypotheses we deduce that for any $k, m \geq 1$ there exists some $\delta_{k, m} > 0$ such that for any $M \geq 1$ and any non-empty open subset $U$ of $X,$
\[
\exists\ x \in U,
\ \forall\ \lambda \in K_m,
\ \exists\ n \geq M,
\ \frac{\# (N_\lambda(x, V_k) \cap [0, n])}{n + 1} > \delta_{k, m}.
\]
For any $k, m \geq 1$ and any $M \geq 1,$ we then define the set $E(k, m, M)$ as
\[
E(k, m, M)
:= \Big\lbrace
x \in X
\ \big\vert \ 
\forall\, \lambda \in K_m,
\, \exists\ n \geq M,
\, \frac{\# (N_\lambda(x, V_k) \cap [0, n])}{n + 1} > \delta_{k, m}
\Big\rbrace.
\]
By the above these sets are clearly dense in $X.$ So in order to prove the claim it is sufficient by the Baire category theorem to show the following assertions:
\begin{itemize}
\item[(a)] $\underset{k, m \geq 1}{\bigcap\,} \underset{M \geq 1}{\bigcap}
E(k, m, M) \subset \underset{\lambda \in \Lambda}{\bigcap} \mathcal{U}FHC(T(\lambda));$
\item[(b)] for any $k, m \geq 1$ and any $M \geq 1,$ the set $E(k, m, M)$ is open.
\end{itemize}
We begin by the proof of (a). Let $x \in X$ be such that
\[
x \in 
\underset{k, m \geq 1}{\bigcap\,} \underset{M \geq 1}{\bigcap}
E(k, m, M).
\]
This then implies that
\[
\forall\ k \geq 1,
\ \forall\ m \geq 1,
\ \forall\ \lambda \in K_m,
\ \denssup (N_\lambda(x, V_k)) > 0.
\]
Since $(K_m)_{m \geq 1}$ covers $\Lambda,$ we obtain that
\[
\forall\ k \geq 1,
\ \forall\ \lambda \in \Lambda,
\ \denssup (N_\lambda(x, V_k)) > 0.
\]
Moreover $(V_k)_{k \geq 1}$ is a base of non-empty open subsets of $X.$ Therefore the previous assertion entails that $x$ is a common upper frequently hypercyclic vector for the family $(T(\lambda))_{\lambda \in \Lambda},$ which proves (a).  So it remains to show (b). Let $k, m, M \geq 1$ and $x \in E(k, m, M).$ By definition we then obtain the existence of a family of positive integers  $(n_\lambda)_{\lambda \in K_m}$ such that
\begin{equation}
\forall\ \lambda \in K_m,
\ n_\lambda \geq M \text{ and }
\ \frac{\# (N_\lambda(x, V_k) \cap [0, n_\lambda])}{n_\lambda + 1}
> \delta_{k, m}.
\label{eq:Basic_criterion_2}
\end{equation}
We must verify that there exists some open neighbourhood $O$ of $x$ such that $O \subset E(k, m, M).$ First we fix $\lambda \in K_m.$ For any $n \in N_\lambda(x, V_k),\ V_k$ is an open neighbourhood of $T^n(\lambda)(x).$ Moreover from Proposition \ref{prop:continuous_family} and Corollary \ref{coro:continuous_family}, we deduce that for each $n \geq 0,$ the map $\Lambda \times X \to X : (\mu, y) \mapsto T^n(\mu)(y)$ is continuous at $(\lambda, x).$ This then implies that for any $n \in N_\lambda(x, V_k),$ there exists some open neighbourhood $U_{n, \lambda}$ of $\lambda$ and some open neighbourhood $O_{n, \lambda}$ of $x$ such that
\[
\forall\ \mu \in U_{n, \lambda},
\ \forall\ y \in O_{n, \lambda},
\ T^n(\mu)(y) \in V_k.
\]
As the set $N_\lambda(x, V_k) \cap [0, n_\lambda]$ is finite, it follows the existence of an open neighbourhood $U_\lambda$ of $\lambda$ and an open neighbourhood $O_\lambda$ of $x$ such that
\[
\forall\ n \in N_\lambda(x, V_k) \cap [0, n_\lambda],
\ \forall\ \mu \in U_\lambda,
\ \forall\ y \in O_\lambda,
\ T^n(\mu)(y) \in V_k.
\]
Altogether we have obtained that for any $\lambda \in K_m,$ there exists some open neighbourhood $U_\lambda$ of $\lambda$ and some open neighbourhood $O_\lambda$ of $x$ such that  
\begin{equation}
\forall\ \mu \in U_\lambda,
\ \forall\ y \in O_\lambda,
\ N_\lambda(x, V_k) \cap [0, n_\lambda]
\subset N_\mu(y, V_k).
\label{eq:Basic_criterion_3}
\end{equation}
In particular $(U_\lambda)_{\lambda \in K_m}$ is a family of open subsets of $\Lambda$ whose union contains $K_m.$ By compactness of $K_m,$ there then exists some $I \geq 1$ and some $\lambda_1, \ldots , \lambda_I \in K_m$ such that
\[
K_m
\subset \underset{1 \leq i \leq I}{\bigcup}
U_{\lambda_i}.
\]
We finally take $O$ to be the open neighbourhood of $x$ given by
\[
O := \underset{1 \leq i \leq I}{\bigcap}
O_{\lambda_i}.
\]
By \eqref{eq:Basic_criterion_3} we then obtain that
\[
\forall\ y \in O,
\ \forall\ \mu \in K_m,
\ \exists\ 1 \leq i \leq I,
\ N_{\lambda_i}(x, V_k) \cap [0, n_{\lambda_i}]
\subset N_\mu(y, V_k).
\]
On the other hand we deduce from \eqref{eq:Basic_criterion_2} that
\[
\forall\ 1 \leq i \leq I,
\ n_{\lambda_i} \geq M
\text{ and }
\ \frac{\# (N_{\lambda_i}(x, V_k) \cap [0, n_{\lambda_i}])}{n_{\lambda_i} + 1}
> \delta_{k, m}.
\]
Combining these assertions we conclude that
\[
\forall\ y \in O,
\ \forall\ \mu \in K_m,
\ \exists\ n > M,
\ \frac{\# (N_\mu(y, V_k) \cap [0, n])}{n + 1}
> \delta_{k, m}.
\]
Therefore the assertion (b) is satisfied. Applying the Baire category theorem we end the proof.
\end{proof}

\section{Common upper frequent hypercyclicity criterion}

The previous theorem gives sufficient conditions to have common upper frequently hypercyclic vectors. In this section we will use this result to establish a theorem in the same vein as the common hypercyclicity theorem from \cite[Chapter 11, Theorem 11.9]{GrPe-11}. Before stating this theorem we have to recall the definition of the unconditional and uniform convergence of a family of series. To this end we will work with an $F$-norm on $X.$  This notion can be found in \cite{KPR-84}.

\begin{definition}
Let $X$ be a vector space. An $F$\textit{-norm} on $X$ is a functional $\Vert \_ \Vert: X \to [0, + \infty[\,$ such that
\begin{itemize}
\item[(i)] for any $x, y \in X,\ \Vert x + y \Vert \leq \Vert x \Vert + \Vert y \Vert;$
\item[(ii)] for any scalar $\lambda$ and any $x \in X,$ if $\vert \lambda \vert \leq 1$ then $\Vert \lambda x \Vert \leq \Vert x \Vert;$ 
\item[(iii)] for any $x \in X,\ \underset{\lambda \to 0}{\lim} \Vert \lambda x \Vert = 0;$ 
\item[(iv)] for any $x \in X,$ if $\Vert x \Vert = 0$ then $x = 0.$ 
\end{itemize}
\end{definition}  

Considering a Fr\'echet space $X$ endowed with a separating increasing sequence of seminorms $(p_k)_{k \geq 1},$ we denote by $\Vert \_ \Vert$ the $F$-norm on $X$ defined, for $x \in X,$ by
\[
\forall\ x \in X,
\ \Vert x \Vert
:= \underset{k \geq 1}{\sum} \frac{1}{2^k} \min (1, p_k(x)). 
\]
This $F$-norm induces the topology of $X$ and allows to work with a norm-like functional.

\begin{definition}
Let $X$ be a Fr\'echet space and $(x_{\lambda, n})_{(\lambda, n) \in \Lambda \times \mathbb{N}}$ a family of vectors of $X.$ We say that the series $\underset{n \geq 1}{\sum} x_{\lambda, n}$ converges \textit{unconditionally and uniformly} for $\lambda \in \Lambda$ if for any $\varepsilon > 0,$ there exists some $N_0 \in \mathbb{N}$ such that 
\[
\text{for any } \lambda \in \Lambda \text{ and any finite set } F \subset \lbrace N_0, N_0 + 1, \ldots \rbrace,
\ \Big\Vert \underset{n \in F}{\sum} x_{\lambda, n} \Big\Vert < \varepsilon.
\]
\end{definition} 

For the application to families of weighted shifts we will use the following sufficient conditions for uniform unconditional convergence from \cite[Chapter 11]{GrPe-11}.

\begin{remark}\label{rem}
Let $X$ be a Fr\'echet space, $(e_n)_{n \geq 1}$ a sequence of vectors of $X$ and $(a_{\lambda, n})_{(\lambda, n) \in \Lambda \times \mathbb{N}}$ a family of scalars. Suppose that there exists some sequence of positive numbers $(c_n)_{n \in \mathbb{N}}$ such that
\begin{itemize}
\item[-] for any $\lambda \in \Lambda$ and any $n \geq 1,
\ \vert a_{\lambda, n} \vert \leq c_n;$
\item[-] the series $\underset{n \geq 1}{\sum} c_n e_n$ converges unconditionally.
\end{itemize}
Then the series $\underset{n \geq 1}{\sum} a_{\lambda, n} e_n$ converges unconditionally and uniformly for $\lambda \in \Lambda.$
\end{remark}

\begin{theorem}[Common upper frequent hypercyclicity criterion]\label{thm:Common_UFHC}
Let $\Lambda$ be a real interval, $X$ a separable Fr\'echet space and $(T(\lambda))_{\lambda \in \Lambda}$ a continuous family of operators on $X.$ Suppose that for any compact subinterval $K$ of $\Lambda,$ there exists some dense subset $X_0$ of $X$ and maps $S_n(\lambda) : X_0 \to X,\ n \geq 0,\ \lambda \in K$ such that for any $x \in X_0,$
\begin{itemize}
\item[\textnormal{(i)}] the series $\underset{n = 0}{\overset{m}{\sum}} T^m(\lambda)(S_{m - n}(\mu_n)(x))$ converges unconditionally and uniformly for $m \geq 0,\ \mu_0 \geq \ldots \geq \mu_m \in K$ and $\lambda \in K;$  
\item[\textnormal{(ii)}] the series $\underset{n \geq 0}{\sum} T^m(\lambda)(S_{m + n}(\mu_n)(x))$ converges unconditionally and uniformly for $m \geq 0, (\mu_n)_{n \geq 0}$ any non-decreasing sequence from $K$ and $\lambda \leq \mu_0 \in K;$ 
\item[\textnormal{(iii)}] for any $\varepsilon > 0,$ there exists some decreasing sequence $(d_n)_{n \geq 1}$ of positive numbers such that
\begin{itemize}
\item[\textnormal{(a)}] for any $n \geq 1$ and any $\lambda, \mu \in K,$
\[
\text{if }
0 \leq \mu - \lambda \leq d_n
\text{ then }
\Vert T^n(\lambda)(S_n(\mu)(x)) - x \Vert \leq \varepsilon;
\]
\item[\textnormal{(b)}] for any $c \in \mathbb{N},$ the series $\underset{t \geq 1}{\sum} d_{c^t}$ diverges;
\end{itemize} 
\item[\textnormal{(iv)}] $(T^n(\lambda)(x))_{n \geq 0}$ converges uniformly to $0$ for $\lambda \in K.$
\end{itemize} 
Then the set of common upper frequently hypercyclic vectors for $(T(\lambda))_{\lambda \in \Lambda}$ is residual in $X,$ and in particular, non-empty.
\end{theorem}

Remark that the finite sums of the hypothesis (i) are regarded as infinite series by adding $0$ terms.

Let us explain the main idea of this result before proving it.

\begin{idea}
In view of Theorem \ref{thm:Basic criterion}, in order to obtain common upper frequent hypercyclicity it is sufficient to construct a vector $z$ from an open subset $U$ which visits frequently enough $V$ under $T(\lambda_t),$ for some $\lambda_0 < \ldots < \lambda_\tau \in \Lambda.$ We create this vector by blocks. The first block will ensure that $z \in U.$ Then in order to visit $V$ we will add perturbations of a vector $y$ in $V.$ More specifically we take $z$ to be of the following form:
\[
z := x +
\underset{t = 1}{\overset{\tau}{\sum}}
(\underset{
\text{to visit } V \text{ under } T(\lambda_t) \ L_t + 1 \text{ times}
}
{
\underbrace{
S_{l_t}(\lambda_t)(y)
+ S_{l_t + s_0}(\lambda_t)(y)
+ \ldots + S_{l_t + L_t s_0}(\lambda_t)(y)
}
}),
\]
where $x \in U$ and for any $t \in \lbrace 1, \ldots, \tau \rbrace,$ we denote by
\begin{itemize}
\item[-] $s_0$ the \textit{gap} between two successive approximations in each block; 
\item[-] $l_t$ the \textit{beginning} of the block for $\lambda_t;$ 
\item[-] $L_t + 1$ the \textit{number of visits} in the block for $\lambda_t.$
\end{itemize}
Notice that we only consider the visiting times with a constant gap. This is a technical necessity. In the proof we will choose $(L_t)_{t \geq 1} := (l_t)_{t \geq 1}$.
\end{idea}

As announced this result is similar to common hypercyclicity theorem from \cite[Chapter 11, Theorem 11.9]{GrPe-11}. The main difference between these theorems is the third hypothesis. For common hypercyclicity we only ask the divergence of the series $\underset{n \geq 1}{\sum} d_n.$ Here we must have the divergence of the series but in a specific way. We select a particular subsequence of $(d_n)_{n \geq 1}.$ This allows to have enough times between approximations for two successive parameters $\lambda.$

\begin{proof}
In order to obtain the claim we will apply Theorem \ref{thm:Basic criterion}. To this end we have to show that the family $(T(\lambda))_{\lambda \in \Lambda}$ satisfies the hypotheses of this result. Let $K$ be a compact subset of $\Lambda.$ We can assume without loss of generality that $K := [a, b]$ is a subinterval of $\Lambda.$ From the hypotheses it then follows the existence of a dense subset $X_0$ of $X$ and maps $S_n(\lambda) : X_0 \to X,\ n \geq 0,\ \lambda \in K$ such that the assertions (i), (ii), (iii) and (iv) are satisfied. Let $V$ be a non-empty open subset of $X.$ Since $X_0$ is dense in $X$ we can suppose that $V := B(y, \varepsilon)$ with some $y \in X_0$ and some $\varepsilon > 0.$  By (iii) there exists some decreasing sequence $(d_n)_{n \geq 1}$ of positive numbers such that
\begin{itemize}
\item[\textnormal{(a)}] for any $n \geq 1$ and any $\lambda, \mu \in K,$
\[
\text{if }
0 \leq \mu - \lambda \leq d_n
\text{ then }
\Vert T^n(\lambda)(S_n(\mu)(y)) - y \Vert \leq \frac{\varepsilon}{4};
\]
\item[\textnormal{(b)}] for any $c \in \mathbb{N},$ the series $\underset{t \geq 1}{\sum} d_{c^t}$ diverges.
\end{itemize}
On the other hand we obtain by (i) and (ii) the existence of a positive integer $s_0$ such that for any finite subset $F$ of $\lbrace s_0, s_0 + 1, \ldots \rbrace$ and any $m \geq 0$ we have that
\begin{itemize}
\item[(I)] for any $\mu_0 \geq \ldots \geq \mu_m \in [a, b]$ and any $\lambda \in [a, b],$
\[
\Big\Vert
\underset{n \in F \cap [0, m]}{\sum}
 T^m(\lambda)(S_{m - n}(\mu_n)(y))
\Big\Vert
< \frac{\varepsilon}{4};
\]
\item[(II)] for any non-decreasing sequence $(\mu_n)_{n \geq 0}$ from $[a, b]$ and any $\lambda \leq \mu_0 \in [a, b],$
\[
\Big\Vert
\underset{n \in F}{\sum}
 T^m(\lambda)(S_{m + n}(\mu_n)(y))
\Big\Vert
< \frac{\varepsilon}{4}.
\]
\end{itemize}
We then take $\delta : = \frac{1}{2 + s_0} > 0.$ Let $U$ be a non-empty open subset of $X$ and $M \geq 1.$ Once again, as $X_0$ is dense in $X$ we can assume that $U := B(x, r)$ with some $x \in X_0$ and some $r > 0.$ From (ii) with $m = 0$ and (iv) we deduce the existence of a positive integer $N_0$ such that for any finite subset $F$ of  $\lbrace N_0, N_0 + 1 \ldots \rbrace,$
\begin{itemize}
\item[(III)] for any non-decreasing sequence $(\mu_n)_{n \geq 0}$ from $[a, b],$
\[
\Big\Vert
\underset{n \in F}{\sum}
 S_n(\mu_n)(y)
\Big\Vert
< r;
\]
\item[(IV)] for any $n \geq N_0$ and any $\lambda \in [a, b],
\ \Vert T^n(\lambda)(x) \Vert < \frac{\varepsilon}{4}.$
\end{itemize}
Furthermore by taking $c := \max (N_0, 2 + s_0, M),$ the assertion (b) ensures that the series
\[
\underset{t \geq 1}{\sum} d_{c^{t + 1}} \text{ diverges}.
\]
Therefore there exists some integer $\tau \geq 1$ such that
\[
a + \underset{t = 1}{\overset{\tau - 1}{\sum}} d_{c^{t + 1}}
\leq b 
< a + \underset{t = 1}{\overset{\tau}{\sum}} d_{c^{t + 1}}.
\]
So we obtain the subdivision $(\lambda_t)_{0 \leq t \leq \tau}$ of $[a, b]$ defined by
\[
\lambda_\tau := b
\text{ and for any } 0 \leq t \leq \tau - 1,
\ \lambda_t := a + \underset{s = 1}{\overset{t}{\sum}} d_{c^{s + 1}}. 
\]
We finally consider the sequence $(l_t)_{t \geq 1} := (c^t)_{t \geq 1}$ and the vector $z$ given by
\[
z := x +
\underset{t = 1}{\overset{\tau}{\sum\,}}
 \underset{l = 0}{\overset{l_t}{\sum}}
  S_{l_t + l s_0}(\lambda_t)(y).
\]
We observe that for any $ 1 \leq t \leq \tau,\ l_t + l_t s_0 < l_{t + 1},$ which implies that the integers $l_t + l s_0,$ with $0 \leq l \leq l_t$ and $t \geq 1$ are pairwise distinct. Combining this with (III) this entails that
\[
\Vert z - x \Vert
= \Big\Vert
\underset{t = 1}{\overset{\tau}{\sum\,}}
 \underset{l = 0}{\overset{l_t}{\sum}}
  S_{l_t + l s_0}(\lambda_t)(y)
\Big\Vert
< r.
\]
Moreover, considering $1 \leq t \leq \tau,\ \lambda \in [\lambda_{t - 1}, \lambda_t]$ and $0 \leq l \leq l_t,$ we have that
\begin{align*}
T^{l_t + l s_0}(\lambda)(z)
& = T^{l_t + l s_0}(\lambda)(x)
+ \underset{s = 1}{\overset{t - 1}{\sum}}
 \underset{k = 0}{\overset{l_s}{\sum\,}}
 T^{l_t + l s_0}(\lambda)(S_{l_s + k s_0}(\lambda_s)(y))\\
+ \underset{k = 0}{\overset{l - 1}{\sum\,}} 
& T^{l_t + l s_0}(\lambda)(S_{l_t + k s_0}(\lambda_t)(y))
+ T^{l_t + l s_0}(\lambda)(S_{l_t + l s_0}(\lambda_t)(y))\\
+ \underset{k = l + 1}{\overset{l_t}{\sum}}
& T^{l_t + l s_0}(\lambda)(S_{l_t + k s_0}(\lambda_t)(y))
+ \underset{s = t + 1}{\overset{\tau}{\sum}}
\underset{k = 0}{\overset{l_s}{\sum\,}}
 T^{l_t + l s_0}(\lambda)(S_{l_s + k s_0}(\lambda_s)(y)),
\end{align*}
which may be written equivalently as
\begin{align*}
T^{l_t + l s_0}(\lambda)(z)
& = T^{l_t + l s_0}(\lambda)(x)
+ T^{l_t + l s_0}(\lambda)(S_{l_t + l s_0}(\lambda_t)(y))\\
& + \bigg(\underset{s = 1}{\overset{t - 1}{\sum}}
\underset{k = 0}{\overset{l_s}{\sum\,}}
 T^{l_t + l s_0}(\lambda)(S_{l_t + l s_0 - (l_t + l s_0 - l_s - k s_0)}(\lambda_s)(y))\\
& \hspace*{5ex} + \underset{k = 0}{\overset{l - 1}{\sum\,}} 
 T^{l_t + l s_0}(\lambda)(S_{l_t + l s_0 - (l_t + l s_0 - l_t - k s_0)}(\lambda_t)(y)) \bigg)\\
& + \bigg( \underset{k = l + 1}{\overset{l_t}{\sum}}
 T^{l_t + l s_0}(\lambda)(S_{l_t + l s_0 + (l_t + k s_0 - l_t - l s_0)}(\lambda_t)(y))\\
& \hspace*{5ex} + \underset{s = t + 1}{\overset{\tau}{\sum}}
\underset{k = 0}{\overset{l_s}{\sum\,}}
 T^{l_t + l s_0}(\lambda)(S_{l_t + l s_0 + (l_s + k s_0 - l_t - l s_0)}(\lambda_s)(y)) \bigg).
\end{align*}
As observed previously, the integers $l_s + k s_0,$ with $0 \leq k \leq l_s$ and $s \geq 1$ are pairwise distinct. Since $\lambda \in [\lambda_{t - 1}, \lambda_t],$ it follows from (I) and (II) that
\[
\Vert T^{l_t + l s_0}(\lambda)(z) - y \Vert
 < \Vert T^{l_t + l s_0}(\lambda)(x) \Vert
+ \Vert T^{l_t + l s_0}(\lambda)(S_{l_t + l s_0}(\lambda_t)(y)) - y \Vert
+ \frac{\varepsilon}{4} + \frac{\varepsilon}{4}.
\]
Moreover by definition, we have that $l_t + l s_0 \geq c = \max (N_0, 2 + s_0, M).$ Whence $l_t + l s_0 \geq N_0.$ From (IV) we then deduce that
\begin{equation}
\Vert T^{l_t + l s_0}(\lambda)(z) - y \Vert
< \frac{3}{4} \varepsilon
+ \Vert T^{l_t + l s_0}(\lambda)(S_{l_t + l s_0}(\lambda_t)(y)) - y \Vert.
\label{eq:Common_UFHC_criterion}
\end{equation}
On the other hand, since $\lambda \in [\lambda_{t - 1}, \lambda_t],$ we have that
\[
0
\leq \lambda_t - \lambda
\leq \lambda_t - \lambda_{t - 1}
\leq d_{c^{t + 1}}
= d_{l_{t + 1}}.
\]
But the sequence $(d_n)_{n \geq 1}$ is decreasing and $l_t + l s_0 < l_{t + 1}.$ This then entails that
\[
0
\leq \lambda_t - \lambda
\leq d_{l_t + l s_0}.
\]
So by (a) we have that
\[
\Vert T^{l_t + l s_0}(\lambda)(S_{l_t + l s_0}(\lambda_t)(y)) - y \Vert
\leq \frac{\varepsilon}{4}.
\]
Together with \eqref{eq:Common_UFHC_criterion} we conclude that
\[
\Vert T^{l_t + l s_0}(\lambda)(z) - y \Vert < \varepsilon.
\]
Altogether we have shown that $z \in U$ and 
\[
\forall\ 1 \leq t \leq \tau,
\ \forall\ \lambda \in [\lambda_{t - 1}, \lambda_t],
\ \forall\ 0 \leq l \leq l_t,
\ T^{l_t + l s_0}(\lambda)(z) \in V.
\]
In particular this implies that
\[
\forall\ 1 \leq t \leq \tau,
\ \forall\ \lambda \in [\lambda_{t - 1}, \lambda_t],
\ \frac{\# (N_\lambda(z, V) \cap [0, l_t + l_t s_0])}{l_t + l_t s_0 + 1}
> \frac{1}{2 + s_0}.
\]
As $l_1 \geq M,$ we have obtained that $z \in U$ and
\[
\forall\ \lambda \in [a, b],
\exists\ n \geq M,
\frac{\# (N_\lambda(z, V) \cap [0, n])}{n + 1}
> \frac{1}{2 + s_0},
\]
which ends the proof.
\end{proof}

\section{Application to weighted shifts}

In this section we give examples of families that have common upper frequently hypercyclic vectors. To this end we will consider families of weighted shifts. These operators form a rich source of examples in linear dynamics.

\begin{definition}
Let $p \geq 1$ and $w= (w_n)_{n \geq 1}$ be a sequence of nonzero scalars. The \textit{weighted shift} $B_w$ is the map on $l^p$ defined, for $x = (x_n)_{n \geq 0} \in l^p,$ by
\[
B_w(x) := (w_{n + 1} x_{n + 1})_{n \geq 0},
\]
where $l^p$ is the Banach space of sequences $(x_n)_{n \geq 0}$ such that
\[
\underset{n \geq 0}{\sum} \vert x_n \vert^p < + \infty.
\]
\end{definition}

We mention that a weighted shift $B_w$ is an operator on $l^p$ if and only if the sequence $w$ is bounded.

More generally we may define these operators on Fr\'echet spaces which possess an unconditional basis.

\begin{definition}
Let $X$ be a Fr\'echet space, $(e_n)_{n \geq 0}$ an unconditional basis of $X$ and $w = (w_n)_{n \geq 1}$ a sequence of nonzero scalars. The \textit{weighted shift} $B_w$ is the map on $X$ defined, for $x = \underset{n \geq 0}{\sum} x_n e_n \in X,$ by
\[
B_w (x) := \underset{n \geq 0}{\sum} w_{n + 1} x_{n + 1} e_n.
\]
Remark that this map is not necessarily well defined.  
\end{definition}

We derive from Theorem \ref{thm:Common_UFHC} the following result for families of weighted shifts. This result is based on a special case of a common hypercyclicity theorem for families of weighted shifts due to Bayart and Matheron \cite{BaMa-07}.

\begin{theorem}\label{thm:Common_UFHC_shifts}
Let $\Lambda$ be a real interval, $X$ a Fr\'echet space, $(e_n)_{n \geq 0}$ an unconditional basis of $X$ and $(w_n(\lambda))_{(n, \lambda) \in \mathbb{N} \times \Lambda}$ a family of positive numbers such that for any $\lambda \in \Lambda,\ B_{w(\lambda)}$ is an operator on $X.$ Suppose that
\begin{itemize}
\item[\textnormal{(i)}] for any $n \geq 1,$ the function $w_n: \Lambda \to \mathbb{R}: \lambda \mapsto w_n(\lambda)$ is increasing;
\item[\textnormal{(ii)}] for any $\lambda \in \Lambda,$ the series
\[
\underset{\nu \geq 1}{\sum}
\frac{1}{w_1(\lambda) \ldots w_\nu(\lambda)} e_\nu
\text{ converges in } X;
\]
\item[\textnormal{(iii)}] for any compact subinterval $K$ of $\Lambda$ and any $n \in \mathbb{N},$ there exists some $L_n(K) > 0$ such that
\[
\forall\ \lambda, \mu \in K,
\ \vert
\log (w_n(\lambda)) - \log (w_n(\mu))
\vert
\leq L_n(K) \vert \lambda - \mu \vert;
\]
\item[\textnormal{(iv)}] for any compact subinterval $K$ of $\Lambda$ and any $s \in \mathbb{N},$ the series
\[
\underset{t \geq 1}{\sum} 
\bigg(
\underset{i = 1}{\overset{s^t}{\sum\,}} L_i(K)
\bigg)^{-1}
\text{ diverges}.
\]
\end{itemize}
Then the set of common upper frequently hypercyclic vectors for $(B_{w(\lambda)})_{\lambda \in \Lambda}$ is residual in $X,$ and in particular, non-empty.
\end{theorem}

\begin{proof}
We will apply Theorem \ref{thm:Common_UFHC}. To this end we must prove that the family $(B_{w(\lambda)})_{\lambda \in \Lambda}$ satisfies the hypotheses of this result. We begin by proving that $(B_{w(\lambda)})_{\lambda \in \Lambda}$ is a continuous family of operators on $X.$ Let $x \in X$ and $K = [a, b]$ be a compact subinterval of $\Lambda.$ As $B_{w(b)}$ is an operator on $X$ and $(e_n)_{n \geq 0}$ is an unconditional basis of $X,$ the series
\[
\underset{n \geq 0}{\sum} w_{n + 1}(b) x_{n + 1} e_n
\]
converges unconditionally. Together with (i) this entails by Remark \ref{rem} that the series
\[
\underset{n \geq 0}{\sum}
w_{n + 1}(\lambda) x_{n + 1} e_n
\]
converges unconditionally and uniformly for $\lambda \in K.$ Using the continuity of maps $\Lambda \to \mathbb{C} : \lambda \mapsto w_n(\lambda),\ n \geq 0,$ we then conclude easily the continuity of the map $\Lambda \to X : \lambda \mapsto B_{w(\lambda)}(x)$ on $K.$ This yields that $(B_{w(\lambda)})_{\lambda \in \Lambda}$ is continuous.

Now we care about the main conditions of Theorem \ref{thm:Common_UFHC}. We fix $K = [a, b]$ a compact subinterval of $\Lambda.$ We take $X_0 := \text{span} \lbrace e_\nu \mid \nu \geq 0 \rbrace$. We know that this space is dense in $X.$ We also consider for any $\lambda \in K,$ the map $F(\lambda)$ on $X_0$ defined by
\[
\forall\ x = \underset{\nu = 0}{\overset{J}{\sum}} x_\nu e_\nu \in X_0,
\ F(\lambda)(x)
= \underset{\nu = 0}{\overset{J}{\sum}}
\frac{1}{w_{\nu + 1}(\lambda)} x_\nu e_{\nu + 1}.
\]
We then take for any $\lambda \in K$ and any $n \geq 0,\ S_n(\lambda) := F^n(\lambda).$ So we have to show that for any $x \in X_0,$ the assertions (I), (II), (III) and (IV) are satisfied.
We observe that for the assertions (I), (II) and (IV) it is sufficient to prove them for $e_\nu,\ \nu \geq 0.$ Firstly we remark that for any nonnegative integer $\nu,$
\begin{equation}
\forall\ \lambda \in K,
\ \forall\ n \geq 0,
\ B^n_{w(\lambda)}(e_\nu)
= \begin{cases}
w_{\nu - n + 1}(\lambda) \ldots w_\nu(\lambda) e_{\nu - n}
& \text{ if } \nu \geq n,\\
0 & \text{ else}.
\end{cases}
\label{eq:Common_UFHC_shifts_1}
\end{equation}
By definition of $(F(\lambda))_{\lambda \in K},$ we also obtain that for any nonnegative integer $\nu,$
\begin{equation}
\forall\ \lambda \in K,
\ \forall\ n \geq 0,
\ F^n(\lambda)(e_\nu)
= \frac{1}{w_{\nu + 1}(\lambda) \ldots w_{\nu + n}(\lambda)} e_{\nu + n}.
\label{eq:Common_UFHC_shifts_2}
\end{equation}
We immediately conclude from \eqref{eq:Common_UFHC_shifts_1} that the assertion (IV) is satisfied for each $e_\nu,\, \nu \geq 0.$ Now we care about the assertions (I) and (II). Let $\nu \geq 0.$ It follows from \eqref{eq:Common_UFHC_shifts_1} and \eqref{eq:Common_UFHC_shifts_2} that
\[
\forall\ \lambda, \mu \in K,
\ \forall\ m \geq n > \nu,
\ B^m_{w(\lambda)}(F^{m - n}(\mu)(e_\nu)) = 0,
\]
which implies (I). Let $m \geq 0,\ (\mu_n)_{n \geq 0}$ be an non-decreasing sequence  from $K$ and $\lambda \leq \mu_0 \in K.$ By \eqref{eq:Common_UFHC_shifts_1} and \eqref{eq:Common_UFHC_shifts_2} we have that
\[
\forall\ n \geq 0,
\ B^m_{w(\lambda)}(F^{m + n}(\mu_n)(e_\nu))
= \frac{w_{\nu + n + 1}(\lambda) \ldots w_{\nu + m + n}(\lambda)}
{w_{\nu + 1}(\mu_n) \ldots w_{\nu + m + n}(\mu_n)}
e_{\nu + n}.
\]
Thus we obtain that 
\begin{equation}
\underset{n \geq 0}{\sum}
B^m_{w(\lambda)}(F^{m + n}(\mu_n)(e_\nu))
= \underset{n \geq 0}{\sum}
\frac{w_{\nu + n + 1}(\lambda) \ldots w_{\nu + m + n}(\lambda)}
{w_{\nu + 1}(\mu_n) \ldots w_{\nu + m + n}(\mu_n)}
e_{\nu + n}.
\label{eq:Common_UFHC_shifts_3}
\end{equation}
Furthermore since $K = [a, b],$ we have by the hypothesis (i) that
\begin{equation}
\forall\ n \geq 0,
\ \Big\vert
\frac{w_{\nu + n + 1}(\lambda) \ldots w_{\nu + m + n}(\lambda)}
{w_{\nu + 1}(\mu_n) \ldots w_{\nu + m + n}(\mu_n)}
\Big\vert
\leq \frac{1}{w_{\nu + 1}(a) \ldots w_{\nu + n}(a)}.
\label{eq:Common_UFHC_shifts_4}
\end{equation}
However by the hypothesis (ii) the series
\[
\underset{n \geq 0}{\sum} \frac{1}{w_{\nu + 1}(a) \ldots w_{\nu + n}(a)} e_{\nu + n}
\text{ converges in } X. 
\]
Combining this with \eqref{eq:Common_UFHC_shifts_3} and \eqref{eq:Common_UFHC_shifts_4} we deduce from Remark \ref{rem} the assertion (II). Finally it remains to show the assertion (III) for each $x \in X_0.$ Let $x \in X_0.$ So, by definition of $X_0,$ there exists some $J \geq 0$ and scalars $x_0, \ldots, x_J$ such that
\[
x := \underset{\nu = 0}{\overset{J}{\sum}} x_\nu e_\nu.
\]
Let $\varepsilon > 0.$ By the hypothesis (iii), there exists some sequence $(L_n)_{n \geq 1}$ of positive numbers such that
\begin{equation}
\forall\ n \geq 1,
\ \forall\ \lambda, \mu \in K,
\ \vert \log (w_n(\lambda)) - \log (w_n(\mu)) \vert
\leq L_n \vert \lambda - \mu \vert.
\label{eq:Common_UFHC_shifts_5}
\end{equation}
On the other hand, by properties of $F$-norms, there exists some $\eta > 0$ such that
\begin{equation}
\forall\ \xi \in \mathbb{R},
\ \Big\lbrack
\vert \xi \vert \leq \eta
\Rightarrow
\Big(
 \forall\ 0 \leq \nu \leq J,
 \ \Vert \xi x_\nu e_\nu \Vert \leq \frac{\varepsilon}{J + 1}
\Big)
\Big\rbrack.
\label{eq:Common_UFHC_shifts_6}
\end{equation}
We take $(d_n)_{n \geq 1}$ the sequence of positive numbers given by
\[
\forall\ n \geq 1,
\ d_n :=
\eta
\bigg(
\underset{i = 1}{\overset{n + J}{\sum}}
L_i
\bigg)^{- 1}.
\]
Now we have to prove the assertions (a) and (b) of (III). By definition of $(d_n)_{n \geq 1},$ the hypothesis (iv) directly gives (b). Let $n \geq 1$ and $\lambda,  \mu \in K$ be such that $0 \leq \mu - \lambda \leq d_n.$ By \eqref{eq:Common_UFHC_shifts_1} and \eqref{eq:Common_UFHC_shifts_2} we obtain that
\begin{equation}
\Vert
B^n_{w(\lambda)}(F^n(\mu)(x)) - x
\Vert
\leq \underset{\nu = 0}{\overset{J}{\sum}}
\Big\Vert
\Big( 
 \frac{w_{\nu + 1}(\lambda) \ldots w_{\nu + n}(\lambda)}
  {w_{\nu + 1}(\mu) \ldots w_{\nu + n}(\mu)} - 1 
\Big)
x_\nu e_\nu
\Big\Vert.
\label{eq:Common_UFHC_shifts_7}
\end{equation}
Moreover as $\lambda \leq \mu$ the hypothesis (i) implies that
\[
\forall\ 0 \leq \nu \leq J,
\ \Big\vert
\frac{w_{\nu + 1}(\lambda) \ldots w_{\nu + n}(\lambda)}
  {w_{\nu + 1}(\mu) \ldots w_{\nu + n}(\mu)} - 1 
\Big\vert
= 1 - \frac{w_{\nu + 1}(\lambda) \ldots w_{\nu + n}(\lambda)}
  {w_{\nu + 1}(\mu) \ldots w_{\nu + n}(\mu)},
\]
which may be written equivalently as
\[
\forall\ 0 \leq \nu \leq J,
\ \Big\vert
\frac{w_{\nu + 1}(\lambda) \ldots w_{\nu + n}(\lambda)}
  {w_{\nu + 1}(\mu) \ldots w_{\nu + n}(\mu)} - 1 
\Big\vert
= 1 -
\expo^{
-\underset{i = 1}{\overset{n}{\sum}}
 \log(w_{\nu + i}(\mu)) -  \log(w_{\nu + i}(\lambda))
}.
\]
In other terms we have that
\[
\forall\ 0 \leq \nu \leq J,
\ \Big\vert
\frac{w_{\nu + 1}(\lambda) \ldots w_{\nu + n}(\lambda)}
  {w_{\nu + 1}(\mu) \ldots w_{\nu + n}(\mu)} - 1 
\Big\vert
= \int^0_{
-\underset{i = 1}{\overset{n}{\sum}}
 \log(w_{\nu + i}(\mu)) -  \log(w_{\nu + i}(\lambda))}
\expo^{\xi} \, \text{d}\xi.
\]
Whence we obtain that
\[
\forall\ 0 \leq \nu \leq J,
\ \Big\vert
\frac{w_{\nu + 1}(\lambda) \ldots w_{\nu + n}(\lambda)}
  {w_{\nu + 1}(\mu) \ldots w_{\nu + n}(\mu)} - 1 
\Big\vert
\leq \underset{i = 1}{\overset{n}{\sum}}
 \big(
  \log(w_{\nu + i}(\mu)) -  \log(w_{\nu + i}(\lambda))
 \big).
\]
Thus by \eqref{eq:Common_UFHC_shifts_5} this entails that
\[
\forall\ 0 \leq \nu \leq J,
\ \Big\vert
\frac{w_{\nu + 1}(\lambda) \ldots w_{\nu + n}(\lambda)}
  {w_{\nu + 1}(\mu) \ldots w_{\nu + n}(\mu)} - 1 
\Big\vert
\leq \bigg(\underset{i = 1}{\overset{n}{\sum}} L_{\nu + i} \bigg)
(\mu - \lambda).
\]
Since $0 \leq \mu - \lambda \leq d_n,$ it then follows from the definition of $d_n$ that
\[
\forall\ 0 \leq \nu \leq J,
\ \Big\vert
\frac{w_{\nu + 1}(\lambda) \ldots w_{\nu + n}(\lambda)}
  {w_{\nu + 1}(\mu) \ldots w_{\nu + n}(\mu)} - 1 
\Big\vert
\leq \eta.
\]
By \eqref{eq:Common_UFHC_shifts_6} and \eqref{eq:Common_UFHC_shifts_7} we then conclude that
\[
\Vert B^n_{w(\lambda)}(F^n(\mu)(x)) - x \Vert
\leq \varepsilon.
\]
Therefore the assertion (a) is satisfied. Altogether we have shown that Theorem~\ref{thm:Common_UFHC} can be applied to the family $(B_{w(\lambda)})_{\lambda \in \Lambda}.$ This implies the claim.
\end{proof}

\begin{corollary}\label{coro:Common_UFHC_Shift}
Let $\Lambda$ be a real interval, $X$ a Fr\'echet space, $(e_n)_{n \geq 0}$ an unconditional basis of $X$ and $(w_n(\lambda))_{(n, \lambda) \in \mathbb{N} \times \Lambda}$ a family of positive numbers such that for any $\lambda \in \Lambda,\ B_{w(\lambda)}$ is an operator on $X.$ Suppose that
\begin{itemize}
\item[\textnormal{(i)}] for any $n \geq 1,$ the function $w_n : \Lambda \to \mathbb{R}: \lambda \mapsto w_n(\lambda)$ is increasing;
\item[\textnormal{(ii)}] for any $\lambda \in \Lambda,$ the series
\[
\underset{\nu \geq 1}{\sum} \frac{1}{w_1(\lambda) \ldots w_\nu(\lambda)} e_\nu
\text{ converges in } X;
\]
\item[\textnormal{(iii)}] for any compact subinterval $K$ of $\Lambda,$ there exists some sequence of positive numbers $(L_n(K))_{n \geq 1}$ such that
\begin{itemize}
\item[\textnormal{(a)}] the series $\underset{n \geq 1}{\sum} L_n(K)$ converges;
\item[\textnormal{(b)}] for any $n \geq 1$ and any $\lambda, \mu \in K,$
\[
\vert \log (w_n(\lambda)) - \log (w_n(\mu)) \vert
\leq L_n(K) \vert \lambda - \mu \vert.
\]
\end{itemize} 
\end{itemize}
Then the set of common upper frequently hypercyclic vectors for $(B_{w(\lambda)})_{\lambda \in \Lambda}$ is residual in $X$, and in particular, non-empty.
\end{corollary}

\begin{proof}
This follows directly from Theorem \ref{thm:Common_UFHC_shifts}.
\end{proof}

As an application of this corollary we obtain the proposition below which gives us a general form of some families with common upper frequently hypercyclic vectors.  

\begin{proposition}\label{prop:Common_UFHC_Shift}
Let $p \geq 1$ and $(a_n)_{n \geq 1},\ (b_n)_{n \geq 1}$ two sequences of nonnegative numbers. We consider the family of sequences $(w(\lambda))_{\lambda \in \mathbb{R}}$ defined by
\[
\forall\ \lambda \in \mathbb{R},
\ w(\lambda) := (a_n \expo^{\lambda b_n})_{n \geq 1}.
\]
Suppose that
\begin{itemize}
\item[\textnormal{(I)}] the sequence $(a_n)_{n \geq 1}$ is a bounded sequence of positive numbers;
\item[\textnormal{(II)}] the series $\underset{n \geq 1}{\sum} b_n$ and $\underset{n \geq 1}{\sum} \Big(\underset{i = 1}{\overset{n}{\prod}} a_i \Big)^{-p}$ converge.
\end{itemize}
Then the set of common upper frequently hypercyclic vectors for $(B_{w(\lambda)})_{\lambda \in \Lambda}$ is residual in $l^p$, and in particular, non-empty.
\end{proposition}

We finally give an example which cannot be obtained with Proposition \ref{prop:Common_UFHC_Shift}. This actually illustrates the fact that the hypotheses of Corollary \ref{coro:Common_UFHC_Shift} are stronger than those of Theorem \ref{thm:Common_UFHC_shifts}. 

\begin{example}
Let $p \geq 1.$ We consider the family of sequences $(w(\lambda))_{\lambda > \frac{1}{p}}$ defined by
\[
\forall\ \lambda > \frac{1}{p},
\ w(\lambda) := 
\Big(\Big(
\frac{n + 1}{n}
\Big)^\lambda
\Big)_{n \geq 1}.
\]
We observe that for any $n \geq 1$ and any $\lambda, \mu > \frac{1}{p},$
\[
\vert
\log (w_n(\lambda)) - \log (w_n(\mu))
\vert
= \log \Big(\frac{n + 1}{n} \Big) \vert \lambda - \mu \vert.
\]
By taking $(L_n(K))_{n \geq 1} := \big(\log \big(\frac{n + 1}{n} \big)\big)_{n \geq 1},$ we deduce from Theorem \ref{thm:Common_UFHC_shifts} that the set of common upper frequently hypercyclic vectors for $(B_{w(\lambda)})_{\lambda > \frac{1}{p}}$ is residual in $l^p.$ 
\end{example}

\section{Common $\mathcal{A}$-hypercyclicity}

We end this paper with a generalization to $\mathcal{A}$-hypercyclicity. This notion, created by B\`es, Menet, Peris and Puig in \cite{BMPP-14}, generalizes hypercyclicity and upper frequent hypercyclicity. To define this we consider Furstenberg families.

\begin{definition}
Let $\mathcal{A}$ be a non-empty family of subsets of $\mathbb{N}_0.$ We say that $\mathcal{A}$ is a \textit{Furstenberg family} if it is \textit{hereditary upward}, that is, if 
\[
\forall\ A \in \mathcal{A},
\ \forall\ B \subset \mathbb{N}_0,
\ (A \subset B \Rightarrow B  \in \mathcal{A}).
\]
Moreover a Furstenberg family $\mathcal{A}$ is called \textit{upper} if $\mathcal{A}$ does not contain the empty set and if there exists some arbitrary set $D,$ some countable set $M$ and some family $(A_{\delta, \mu})_{(\delta, \mu) \in D \times M}$  of subsets of $\mathbb{N}_0$ such that $\mathcal{A}$ can be written as
\[
\mathcal{A} = \underset{\delta \in D}{\bigcup} \mathcal{A}_\delta
\text{ with } \mathcal{A_\delta} := \underset{\mu \in M}{\bigcap} A_{\delta, \mu} 
\]
and satisfying the following properties:
\begin{itemize}
\item[\textnormal{(i)}] $\mathcal{A}$ is \textit{uniformly left-invariant}, that is,
\[
\forall\ A \in \mathcal{A},
\ \exists\ \delta \in D,
\ \forall\ k \in \mathbb{N}_0,
\ (A - k) \cap \mathbb{N}_0 \in \mathcal{A}_\delta;
\]
\item[\textnormal{(ii)}] for any $\delta \in D$ and any $\mu \in M,$ the family $A_{\delta, \mu}$ is \textit{finitely hereditary upward}, that is,
\[
\forall\ A \in \mathcal{A}_{\delta, \mu},
\ \exists\ F \subset \mathbb{N}_0 \text{ finite},
\ \forall\ B \subset \mathbb{N}_0,
\ ( A \cap F \subset B \Rightarrow B \in A_{\delta, \mu}).
\]
\end{itemize}
\end{definition}

\begin{definition}
Let $X$ be a Fr\'echet space, $T$ an operator on $X$ and $\mathcal{A}$ a Furstenberg family. We say that $T$ is \textit{$\mathcal{A}$-hypercyclic} on $X$ if there exists some vector $x \in X$ such that
\[
\text{for any non-empty open subset } V \text{ of } X,
\ \lbrace n \geq 0 \mid T^n(x) \in V \rbrace \in \mathcal{A}.
\]
In this case, $x$ is called \textit{$\mathcal{A}$-hypercyclic} for $T.$ 
\end{definition}

Actually Bonilla and Grosse-Erdmann have obtained in \cite{BoGr-16} an analogue to the Birkhoff theorem for $\mathcal{A}$-hypercyclicity with $\mathcal{A}$ an upper Furstenberg family. Using similar arguments as in the proof of Theorem \ref{thm:Basic criterion}, we establish the following generalization of Theorem \ref{thm:Basic criterion} for common $\mathcal{A}$-hypercyclicity.

\begin{theorem}
Let $\mathcal{A}$ be an upper Furstenberg family written as
\[
\mathcal{A} :=
\underset{\delta \in D}{\bigcup}
\underset{\mu \in M}{\bigcap}
\mathcal{A}_{\delta, \mu}.
\]
Let $\Lambda$ be a $\sigma$-compact metric space, $X$ a separable Fr\'echet space and $(T(\lambda))_{\lambda \in \Lambda}$ a continuous family of operators on $X.$ Suppose that for any non-empty open subset $V$ of $X$ and for any compact subset $K$ of $\Lambda$ there exists some $\delta \in D$ such that, for any non-empty open subset $U$ of $X$ and any $\mu \in M,$ we have
\[
\exists\ x \in U,
\ \forall\ \lambda \in K,
\ N_\lambda(x, V) \in \mathcal{A}_{\delta, \mu}. 
\]
Then the set of common $\mathcal{A}$-hypercyclic vectors for $(T(\lambda))_{\lambda \in \Lambda}$ is residual in $X.$
\end{theorem}

This theorem can allow to adapt results of Sections 2 and 3 to common $\mathcal{A}$-hypercyclicity. So we can obtain common $\mathcal{A}$-hypercyclicity criteria for some families $\mathcal{A}$ (see \cite{M-2018}).

\bibliographystyle{plain}

{\let\thefootnote\relax\footnote{{D\'epartement de Math\'ematique,
Institut Complexys, Universit\'e de Mons,
20 Place du Parc, 7000 Mons, Belgium}}}
{\let\thefootnote\relax\footnote{{E-mail address: monia.mestiri@umons.ac.be}}}

\end{document}